\documentclass{article}

\usepackage{url}

\usepackage{amsthm}
\usepackage{amsmath}
\usepackage{amssymb}
\usepackage{amsfonts}
\usepackage{enumitem}
\usepackage{cite}

\usepackage{adjustbox}
\usepackage {mathpartir}

\theoremstyle{definition}
\newtheorem{example}{Example}[section]
\newtheorem{definition}[example]{Definition}

\theoremstyle{definition}
\newtheorem{remark}[example]{Remark}

\theoremstyle{plain}
\newtheorem{theorem}[example]{Theorem}
\newtheorem{corollary}[example]{Corollary}
\newtheorem{prop}[example]{Proposition}
\newtheorem{lemma}[example]{Lemma}

\newcommand{\barrightarrow}{\relbar \joinrel \mapstochar \joinrel \rightarrow}
\newcommand{\hto}{\not \to}
\newcommand{\longhto}{\not \longrightarrow}

\newcommand{\pd}{\mathbf{pd}}
\newcommand{\el}{\operatorname{el}}
\newcommand{\id}{\operatorname{id}}
\newcommand{\glob}{\mathbb{G}}
\newcommand{\globSet}{\glob\text{-Set}}

\newcommand{\set}{\operatorname{Set}}
\newcommand{\op}{{\operatorname{op}}}
\newcommand{\nat}{\mathbb{N}}

\newcommand{\globGraph}{{\operatorname{GlobGraph}}}
\newcommand{\rGlobGraph}{\overline{\operatorname{GlobGraph}}}

\newcommand{\spanBatanin}{\operatorname{Span}}
\newcommand{\spanT}{\operatorname{Span}_{T}}

\newcommand{\globPlus}{\glob^{+}}
\newcommand{\limit}{\operatorname{lim}}

\newcommand{\IdType}{\mathcal{H}}

\newcommand{\Type}{\operatorname{Type}}

\newcommand{\monGlobCat}{{\operatorname{MonGlobCat}}}
\newcommand{\globMult}{{\operatorname{GlobMult}}}

\newcommand{\globMultId}{\operatorname{GlobMult}_{\IdType}}
\newcommand{\globMultStrId}{\operatorname{GlobMult}_{\IdType_{\operatorname{Str}}}}
\newcommand{\StrOmegaCat}{{\operatorname{Str}\omega\operatorname{-Cat}}}
\newcommand{\globT}{\mathbb{G}_T(\mathcal{T})}
\newcommand{\typeTheoryT}{\mathcal{T}}

\newcommand{\arrowTheoryT}{\mathbb{T}_0}

\newcommand{\FreeH}{F_{\IdType}}
\newcommand{\FreeCat}{{F_{\operatorname{Cat}}}}
\newcommand{\FreeCatStr}{{G}}

\newcommand{\strictify}{S}

\newcommand{\End}{\operatorname{End}}

\newcommand{\etaStr}{\eta}
\newcommand{\sourceContext}{d}
\newcommand{\outputType}{c}
\newcommand{\arrowTheory}{\mathbb{A}}

\newcommand{\boundaryMap}[2]{[{#1}, {#2}]}

\newcommand{\interpret}[1]{{#1}^{\circ}}

\usepackage{tikz-cd}

\tikzset{%
    symbol/.style={%
    draw=none,
        every to/.append style={%
        edge node={node [sloped, allow upside down, auto=false]{$#1$}}}
      }
    }

\title{Globular Multicategories with Homomorphism Types}
\author{Christopher J. Dean}
\date{}

\begin{document}

\maketitle

\abstract{
  We introduce a notion of globular multicategory with homomorphism types. These structures arise when organizing collections of ``higher category-like'' objects such as type theories with identity types. We show how these globular multicategories can be used to construct various weak higher categorical structures of types and terms.
}

\section{Introduction}

Suppose that we have a dependent type theory with identity types $\mathcal{T}$. Then both van den Berg and Garner \cite{Berg2011Types-are-weak-} and  Lumsdaine \cite{Lumsdaine2010Weak-omega-cate} have shown that the tower of identity types of each type in $\mathcal{T}$ has the structure of a Batanin-Leinster weak $\omega$-category (see \cite{Batanin1998Monoidal-Globul} and \cite{LeinsterHigher-Operads-}). As an intermediate step van den Berg and Garner construct a monoidal globular category of contexts and terms from $\mathcal{T}$.
Our goal in this paper is to understand and generalize this phenomenon.

In section \ref{backgroundSection}, we review Leinster's theory of generalized multicategories \cite{LeinsterHigher-Operads-} and focus on the free strict $\omega$-category monad on globular sets.
The associated notion of generalized multicategories are called \emph{globular multicategories}.
In the terminology of \cite{Cruttwell2010A-unified-frame} these are the ``virtual''  analogues of monoidal globular categories.)

We then describe a notion of homomorphism type within a globular multicategory in section \ref{sectionHomomorphismTypes}.
Every type theory with identity types gives rise to such a globular multicategory.
However, we also capture \emph{directed} examples such as the double category of categories, functors, profunctors and transformations.
In general we believe that objects in a globular multicategory with homomorphism types can be seen as ``higher category-like''.
Future work will make comparisons with structures arising in the study of directed homotopy type theory.

Section \ref{homotopy} is a brief interlude in section \ref{homotopy} to study the homotopy theory of globular multicategories. We develop the tools we require to study weak higher categorical structures in this setting,

Finally in section \ref{sectionWeakHigherCategories} we show that when a globular multicategory has homomorphism types,
there is a precise sense in which:
\begin{itemize}
  \item types behave like weak higher categories
  \item dependent types behave like profunctors
  \item terms behave like higher functors and transformations between profunctors
  \item the collection of types and terms has the structure of a weak $\omega$-category.
\end{itemize}
These results generalize those of van den Berg and Garner \cite{Berg2011Types-are-weak-} and Lumsdaine \cite{Lumsdaine2010Weak-omega-cate, LumsdaineHigher-categori}.
Weak higher categorical structures based on globular multicategories have previously been studied by Kachour. (See for instance \cite{Kachour2015Operads-of-high}.)

\section{Background} \label{backgroundSection}
\subsection{Globular Pasting Diagrams}

\begin{definition}
  The category of  \textbf{globes} $\glob$ is freely generated by the morphisms
  \begin{equation*}
    \begin{tikzcd}
      0 \ar[shift left]{r}{\sigma_0} \ar[swap, shift right]{r}{\tau_0}&
      1 \ar[shift left]{r}{\sigma_1} \ar[swap, shift right]{r}{\tau_1}&
      \cdots
      \ar[shift left]{r}{\sigma_{n-1}} \ar[swap, shift right]{r}{\tau_{n-1}}&
      n \ar[shift left]{r}{\sigma_n} \ar[swap, shift right]{r}{\tau_n}&
      \cdots
    \end{tikzcd}
  \end{equation*}
  subject to the \textbf{globularity} conditions
  \[
    \begin{aligned}
      \sigma_{n+1} \circ \sigma_n &= \tau_{n+1} \circ \sigma_n\\
      \sigma_{n+1} \circ \tau_{n} &= \tau_{n+1} \circ \tau_{n}
    \end{aligned}
  \]
  which ensure that for each $k < n$,
  there are exactly two composite arrows in $\glob$
  of the form
  \[
    \begin{aligned}
      \sigma_k : k \longrightarrow n\\
      \tau_k : k \longrightarrow n
    \end{aligned}
  \]

  A \textbf{globular object} in a category $\mathcal{C}$ is a functor $ A : \glob^{\op} \longrightarrow \mathcal{C}$.
  We denote the image of $\sigma_n : n \to n+1$ under such a functor by
  \begin{equation*}
    \begin{tikzcd}
      A(n)
      \ar[shift left]{r}{s_{k}} \ar[swap, shift right]{r}{t_{k}}&
      A(k)
    \end{tikzcd}
  \end{equation*}
  A presheaf on $\glob$ is called \textbf{globular set}.
  In this case,
  we refer to the elements of $A(n)$ as $n$-\textbf{cells}
  and think of the maps $s_k, t_k$ as describing \textbf{source} and \textbf{target} $k$-cells.
\end{definition}

We now define a notion of $n$-\textbf{pasting diagram} for each $n$.
We will inductively define a globular set $\pd$ whose $n$-cells are $n$-pasting diagrams.

\begin{definition}
  We define  the globular set $\pd$
  together with, for each $k < n$ and each $\pi \in \pd(n)$,
  boundary inclusions
  \[
    \begin{aligned}
      \sigma_k : s_k \pi &\longrightarrow \pi\\
      \tau_k : t_k \pi &\longrightarrow \pi
    \end{aligned}
  \]
  inductively:
  \begin{itemize}
    \item Every globe $n$,
    identified with its image under the Yoneda embedding,
    is an $n$-pasting diagram.
    In this case $\sigma_k$ and $\tau_k$ are induced by the
    corresponding maps in $\glob$.

    \item Whenever $\pi$ is an $n$-pasting diagram,
    there is an $(n+1)$-pasting diagram $\pi^+$
    whose underlying globular set is just $\pi$.
    We define the boundary inclusions to be the identity map:
    \[
      \sigma_k = \tau_k = \id_\pi.
    \]

    \item
    Given $n$-pasting diagrams $\pi_1, \pi_2$ such that $t_k \pi_1 = s_k \pi_2 = \rho$
    the pushout $\pi_1 +_k \pi_2$ of the following diagram
    \begin{equation*}
      \begin{tikzcd}
        \rho \ar{r}{\tau_k}
        \ar[swap]{d}{\sigma_k}
        &
        \pi_1
        \ar{d}{}
        \\
        \pi_2
        \ar[swap]{r}{}
        &
        \pi_1 +_k \pi_2
      \end{tikzcd}
    \end{equation*}
    is a globular pasting diagram.
    In other words we can  ``glue pasting diagrams along shared $k$-cells''.
    When $j < k$,
    we define $\sigma_j$, $\tau_j$
    to be the composites:
    \[
      \begin{aligned}
        s_j \rho &\overset{\sigma_j}{\longrightarrow} \rho
        &\longrightarrow \pi_1 +_k \pi_2
        \\
        t_j \rho &\overset{\tau_j}{\longrightarrow} \rho
        &\longrightarrow \pi_1 +_k \pi_2
      \end{aligned}.
    \]
    When $j = k$,
    we define $\sigma_j$,
    $\tau_j$
    to be the composites
    \[
      \begin{aligned}
        s_j \pi_1 &\overset{\sigma_j}{\longrightarrow} \pi_1
        &\longrightarrow \pi_1 +_k \pi_2\\
        t_j \pi_2 &\overset{\tau_j}{\longrightarrow} \pi_2
        &\longrightarrow \pi_1 +_k \pi_2
      \end{aligned}
    \]
    When $j > k$,
    we define
    \[
      \begin{aligned}
        s_j (\pi_1 +_k \pi_2) &= (s_j \pi_1) +_k (s_j \pi_2)\\
        t_j (\pi_1 +_k \pi_2) &= (t_j \pi_1) +_k (t_j \pi_2)
      \end{aligned}
    \]
    and $\sigma_j, \tau_j$
    are induced by the universal properties of these pushouts.
  \end{itemize}

  This notion of pasting diagram allows us to describe the free strict $\omega$-category monad $T : \globSet \to \globSet$ explicitly.
  For any globular set $A$,
  an $n$-cell of $TA$ consists of
  a pasting diagram $\pi \in \pd(n)$
  together with a map
  \[
    \begin{aligned}
      \pi \to A
    \end{aligned}.
  \]
  Thus, we refer to these maps as \textbf{pasting diagrams} in $A$.
  Suppose that $f, g \in TA(n)$
  are $n$-cells such that such that $t_k f = s_k g$.
  Then we have corresponding pasting diagrams
  \[
    \begin{aligned}
      f : \pi_1 \to A\\
      g : \pi_2 \to A
    \end{aligned}
  \]
  such that $f \tau_k = g \sigma_k$.
  We define $f \odot_{k} g$ to be the induced map
  \[
    \begin{aligned}
      \pi_1 +_{k} \pi_2 \to A
    \end{aligned}.
  \]
  This defines another element of $TA(n)$.
  More generally
  suppose that we have a pasting diagram $\Gamma : \rho \to TA$.
  Then the multiplication of $T$
  gives us a new pasting diagram which we denote
  \[
    \begin{aligned}
      \bigodot_{i \in \rho} \Gamma_i
    \end{aligned}.
  \]
  Suppose that $\rho$ is the $n$-globe, thought of as an $k$-pasting diagram for some $k \geq n$.
  Then $\Gamma$ is uniquely determined by the unique $n$-cell $M \in \Gamma(n)$.
  Hence we write
  \[
    \begin{aligned}
      \Gamma = [M]
    \end{aligned}
  \]
  and we have that
  \[
    \begin{aligned}
      \bigodot_{i \in \rho} \Gamma_i = M
    \end{aligned}.
  \]
\end{definition}

A crucial property of the monad $T$ is that it is \emph{cartesian}; i.e.,
its underlying functor preserves pullbacks
and the naturality squares of its unit and multiplication are pullback squares
(see  \cite{LeinsterHigher-Operads-}).
Following Leinster \cite{LeinsterHigher-Operads-},
this allows us to define a notion of generalized multicategory.

\subsection{Globular Multicategories}

We will assume for the rest of this paper that $T$ is the free strict $\omega$-category monad.
In this section we  review Leinster's \cite{LeinsterHigher-Operads-} theory of $T$-multicategories.

\begin{definition}
  A \textbf{$T$-span} is a span of the following form.
  
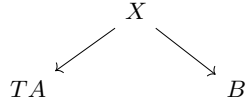
\begin{equation*}
    \begin{tikzcd}
      & X \ar{dl} \ar{dr}\\
      TA && B
    \end{tikzcd}
  \end{equation*}
  We can compose $T$-spans
  
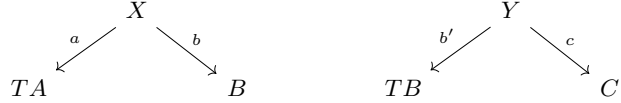
\begin{equation*}
   \begin{tikzcd}
    & X \ar[swap]{dl}{a}\ar{dr}{b}& &&& Y\ar[swap]{dl}{b'}\ar{dr}{c}\\
    TA && B && TB && C\\
  \end{tikzcd}
\end{equation*}
by computing a pullback as in the following diagram
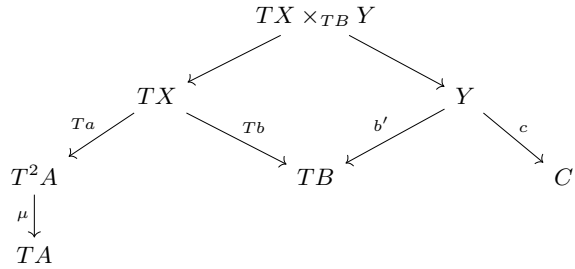
\begin{equation*}
 \begin{tikzcd}
  && TX \times_{TB} Y \ar{dl}\ar{dr}\\
  & TX \ar[swap]{dl}{Ta}\ar{dr}{Tb}& & Y\ar[swap]{dl}{b'}\ar{dr}{c}\\
  T^{2}A \ar[swap]{d}{\mu}&& TB  && C\\
  TA
\end{tikzcd}
\end{equation*}
Let $\eta_X$ be the unit of $T$ at $X$.
Then the identity $T$-span at $X$ is the following diagram:
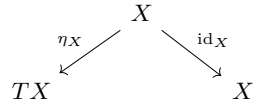
\begin{equation*}
 \begin{tikzcd}
  & X \ar[swap]{dl}{\eta_X}\ar{dr}{\id_X}&\\
  TX && X
\end{tikzcd}
\end{equation*}
Putting all this data together we obtain a bicategory $T$\textbf{-Span}.
\end{definition}

This allows us to present our main definition succinctly.

\begin{definition}
  A \textbf{globular multicategory} is a monad in the bicategory $T$\textbf{-Span}.
\end{definition}

Suppose that $X$ is a globular multicategory.
Then there is a $T$-span of the following form:
\begin{equation*}
  \begin{tikzcd}
    & X_1 \ar[swap]{dl}{\sourceContext} \ar{dr}{\outputType}\\
    TX_0 && X_0
  \end{tikzcd}
\end{equation*}
We refer to these $T$-spans as \textbf{globular multigraphs} .
We will use type theoretic terminology to refer to the data contained in $X$:

\begin{itemize}
  \item An $n$-\textbf{type} is an element of $X_0(n)$.
  When $sM = A$ and $tM$ = B,
  we write
  \[
    \begin{aligned}
      M : A \barrightarrow B
    \end{aligned}.
  \]

  \item A $\pi$-shaped \textbf{$n$-context} $\Gamma = (\Gamma_i)_{i \in \el(\pi)}$ is an element of $TX_0 (n)$ of the form
  \[
    \Gamma : \pi \longrightarrow X_0
    .\]
    A \textbf{variable} in $\Gamma$ is an element $x \in \el(\pi)$.
    When $A = \Gamma_x$, we say that $A$ is the \emph{type} of $x$ and write
    \[
      \begin{aligned}
        x : A
      \end{aligned}.
    \]

    \item An $n$-\textbf{term} $f$ is an element of $X_1(n)$.
    Suppose that $\Gamma = \sourceContext f$ is a $\pi$-shaped $n$-context.
    Let $A = \outputType f$ be an $n$-type.
    Then $f$ can be thought of as a generalized arrow sending
    a $\pi$-shaped $n$-context $\Gamma$ (a pasting diagrams of typed input variables)
    to an $n$-type $A$.
    For this reason we say that $f$ is $\pi$-shaped and write
    \[
      f : \Gamma \longrightarrow A
    \]
    When $n \geq 0$, $n$-terms also have source and target $(n-1)$-terms $sf$ and $tf$.

    \item A \textbf{substitution} or \emph{context morphism} is an element $f \in TX_1(n).$
    For some $\pi \in \pd(n)$,
    $f$ is a collection of terms
    \[
      f_i : \Gamma_i \longrightarrow \Delta_i
    \]
    for $i \in \el(\pi)$.
    (That is a $\pi$-shaped pasting diagram of terms.)
    By pasting together the domain contexts of these terms we obtain a context
    \[
      \begin{aligned}
        \Gamma = \bigodot_{i \in \pi} \Gamma_i
      \end{aligned}
    \]
    By pasting together the codomain types we obtain a $\pi$-shaped context
    \[
      \begin{aligned}
        \Delta = \bigodot_{i \in \pi} \Delta_i
      \end{aligned}
    \]
    In this case,
    we write
    \[
      f : \Gamma \longrightarrow \Delta
    \].

    \item Suppose that we have a substitution $f : \Gamma \to \Delta$
    and a term $g : \Delta \to A$.
    Then the multiplication of $X$
    allows us to define a \textbf{composite} term
    \[
      f ; g : \Gamma \longrightarrow \Delta
      .\]
      We think of $f ; g$
      as the result of substituting $f_i$
      for the variable $i$ into (the domain context of) $g$.
      Now suppose that $h : \Delta \to E$ is a substitution.
      Then we define
      \[
        \begin{aligned}
         f ; h : \Gamma \longrightarrow E
       \end{aligned}
     \]
     to be the substitution such that
     \[
      (f ; h)_j = (f_{i_j})_{i_j \in d h_j} ; h_j
    \].

    \item  For each $n$-type $A$,
    the unit of $X$ gives us
    an \textbf{identity} $n$-term
    \[
      \begin{aligned}
        \id_A : [A] \to A
      \end{aligned}
      .\]
      For each context $\Gamma$,
      we define
      \[
        \begin{aligned}
          \id_\Gamma : \Gamma \to \Gamma
        \end{aligned}
      \]
      to be the substitution such that
      \[
        \begin{aligned}
          (\id_\Gamma)_i = \id_{\Gamma_i}
        \end{aligned}
      \]
      We think of this as the \emph{trivial substitution}.

      \item Associativity of $X$ tells us that,
      whenever it makes sense,
      \[
        \begin{aligned}
          (f ; g) ; h = f ; (g ; h)
        \end{aligned}
      \]
      \item The unit laws of $X$ tell us that,
      whenever it makes sense,
      \[
        \begin{aligned}
          f ; \id_A = f = \id_\Gamma ; f
        \end{aligned}
      \]
    \end{itemize}

    \begin{example}
      Batanin's \cite{Batanin1998Monoidal-Globul} globular operads are a particular important class of globular multicategories. A \textbf{globular operad} is a globular multicategory $X$ with $X_0 = \top$, the terminal globular set.
      In other words a globular operad has a unique $n$-type for each $n \in \glob$.

      The terminal globular multicategory (operad) has a unique $\pi$-shaped $n$-term for each $\pi \in \pd(n)$,.
      Algebras of this operad  are closely related to strict $\omega$-categories.
    \end{example}

    \begin{example}
      Let $\mathcal{C}$ be a category with pullbacks.
      There is a globular multicategory $\spanT(\mathcal{C})$ such that:

      \begin{itemize}
        \item An $n$-type is a functor $A : \el(n)^{\op} \to \mathcal{C}$ from the category of elements of the representable globular set $n$.

        \item Given a pasting diagram $\pi \in \pd(n)$,
        a context with shape $\pi$ is a functor
        \[
          \Gamma : \el(\pi)^{\op} \longrightarrow \mathcal{C}
          .\]
          Associated to such a context there is a canonical functor $\Gamma_0 : \el(n)^{\op} \to A$, which sends
          \begin{align*}
            s : m \to n \in \glob &&\longmapsto&&
            \
            \limit \left(
            \el(s \pi)^{\op} \overset{\el(s)^{\op}}{\longrightarrow}
            \el(\pi)^{\op} \overset{\Gamma}{\longrightarrow}
            \mathcal{C}
            \right)
          \end{align*}
          and sends arrows to the canonical morphisms induced between these limits.

          \item A term $u : \Gamma \to A$ in $\spanT(\mathcal{C})$ is a natural transformation $\Gamma_0 \to A$.
        \end{itemize}
      \end{example}

      \begin{remark}
        We refer to functors $\el(n)^{\op} \to \mathcal{C}$ as \textbf{$n$-spans} in $\mathcal{C}$ (see \cite{Batanin1998Monoidal-Globul}).
        The globular multicategory $\spanT(\mathcal{C})$ underlies the monoidal globular category $\spanBatanin(\mathcal{C})$ described ibid.
        A $1$-span is a span in the usual sense.
        More generally,
        an $(n+1)$-span is a ``span between $n$-spans''.
        For example a $3$-span is a diagram of the following form:
        
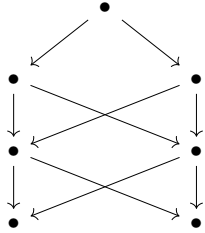
\begin{equation*}
          \begin{tikzcd}
            & \bullet \ar{dl} \ar{dr} &
            \\
            \bullet \ar{d} \ar{drr} &&
            \bullet \ar{d} \ar{dll}
            \\
            \bullet \ar{d} \ar{drr}&&
            \bullet \ar{d} \ar{dll}
            \\
            \bullet &&
            \bullet
          \end{tikzcd}
        \end{equation*}
      \end{remark}

      \begin{definition}
        A \textbf{map} of globular multigraphs is a pair of arrows $f_0, f_1$ making the following diagram commute:
        
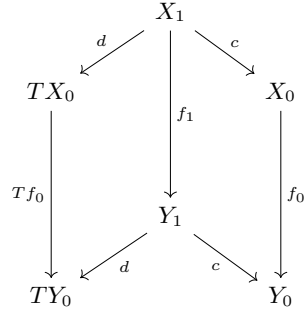
\begin{equation*}
          \begin{tikzcd}
            & X_1\ar[swap]{dl}{\sourceContext}\ar{dr}{\outputType}\ar{ddd}{f_1} &\\
            TX_0 \ar[swap]{ddd}{Tf_0} && X_0\ar{ddd}{f_0}\\
            \\
            & Y_1\ar{dl}{\sourceContext}\ar[swap]{dr}{\outputType} &\\
            TY_0 && Y_0\\
          \end{tikzcd}
        \end{equation*}
        A \textbf{homomorphism} of globular multicategories is a map of globular multigraphs preserving the multiplication and unit of $X$.
        We denote the category of globular multicategories and homomorphisms by $\globMult$.
      \end{definition}

      A globular multicategory can be also be seen as an \emph{algebraic theory} whose operations have arities which are shapes of pasting diagrams.
      For this reason we also think of homomorphisms $X \to Y$ as \emph{algebras} of $X$ in $Y$.
      \begin{remark}
        Algebras of the terminal globular operad in $\spanT(\mathcal{C})$ are \emph{strict $\omega$-categories internal to $\mathcal{C}$}.
        In particular,
        algebras of the terminal globular operad in $\spanT(\set)$ are strict $\omega$-categories.
      \end{remark}

      \begin{definition}
        Let  $X : \glob^\op \to \mathcal{C}$ be a globular object in $\mathcal{C}$.
        The \textbf{endomorphism operad} $\End(X)$
        is the collection of terms in $\spanT{\mathcal{C}}$
        whose types come from $X$
        (see \cite{Batanin1998Monoidal-Globul}).
        To be precise,
        let $\hat X : \globSet \to \mathcal{C}$ be the canonical cocontinuous extension of $X$.
        Then a term $f : \pi \to n$ in $\End(X)$
        is a term
        \[
          \begin{aligned}
            \hat X(\pi) \longrightarrow X(n)
          \end{aligned}
        \]
        in $\spanT(\mathcal{C})$ respecting boundaries.
        This is the objects part of a functor
        \[
          \begin{aligned}
            \End : \mathcal{C}^{\glob^\op} \longrightarrow \globMult
          \end{aligned}
        \]

      \end{definition}

      Every globular multigraph can be viewed as a presheaf over a category whose objects are shapes of types and terms.

      \begin{definition}
        The category $\globPlus$ of \textbf{generic} types and terms is defined as follows:
        Its set of objects is the the coproduct of sets
        \begin{align*}
          \glob + \el(\pd)
        \end{align*}.
        In this case, we say that $n \in \glob$ is the \emph{generic n-type}
        and that $\pi \in \pd(n)$ is the \emph{generic $\pi$-shaped $n$-term}.
        There are four classes of arrows in $\globPlus$:
        \begin{itemize}
          \item Every arrow in $\glob$ induces a corresponding arrow between generic types in $\globPlus$.
          These arrows pick out the source and target types of generic types.

          \item Every arrow $\el(\pd)$ induces a corresponding arrow between generic terms in $\globPlus$.
          These arrows pick out the source and target terms of generic terms.

          \item Let $u$ be the generic $\pi$-shaped $n$-pasting diagram
          and let $k \in \glob$.
          Then each map of globular sets $k \to \pi$
          induces an arrow
          \[
            \begin{aligned}
              k \to u
            \end{aligned}
          \]
          in $\globPlus$ from the generic $k$-type to the generic $\pi$-shaped term $u$.
          These arrow pick out the types in the domain contexts of generic terms.

          \item  For each generic $n$-term $u$
          and each arrow $k \to n$ in $\glob$
          there is an arrow
          \[
            \begin{aligned}
              k \to u
            \end{aligned}
          \]
          from the generic $k$-type to the generic $n$-term $u$.
          These arrows pick out the codomain types of generic terms.
        \end{itemize}
      \end{definition}

      \begin{prop}
        The category of globular multigraphs $\globGraph$ is equivalent to the category of presheaves over $\globPlus$.
      \end{prop}
      \begin{proof}
        This follows from  \cite[Proposition C.3.4 and Proposition 6.5.6]{LeinsterHigher-Operads-}.
      \end{proof}

      \begin{remark}\label{globPlusDirect}
      The category $\globPlus$ is a \emph{direct category}.
      There is an identity-reflecting functor $\dim : \globPlus \to \nat$ which sends
      the generic $n$-types and all generic $n$-terms to the natural number $n$.
      Let $X$ be a globular multigraph.
      Let $u \in \globPlus$ be a generic type or term,
      identified with its image under the Yoneda embedding.
      Then the \textbf{boundary} $\partial u$ is the subpresheaf of $u$ such that
      \[
        \begin{aligned}
          w \in \partial u(v) \iff \dim v < \dim u
        \end{aligned}.
      \]
      We will denote the canonical inclusion by
      \[
        \begin{aligned}
          \iota_u : \partial u \longrightarrow u
        \end{aligned}.
      \]
      In Section \ref{homotopy} we will use these boundary inclusions to describe a higher dimensional notion of weakness.
    \end{remark}

    \section{Homomorphism Types}\label{sectionHomomorphismTypes}

    \begin{definition}
      We say that a globular multigraph is \textbf{reflexive}
      when for each $n$-type $A$,
      we have an ``identity'' $(n+1)$-type
      \[
        \begin{aligned}
          \IdType_A : A \barrightarrow A
        \end{aligned}
      \]
      and a \textbf{reflexivity $(n+1)$-term}
      \[
        \begin{aligned}
          r_A : A \to \IdType_A
        \end{aligned}
      \]
    \end{definition}

    \begin{definition}
      Let $0 \leq k < n$.
      Suppose that $\Gamma$ is a $\pi$-shaped $n$-context
      in a reflexive globular multigraph.
      Then for any $k$-variable $x : A$ in $\Gamma$,
      since $\IdType_A : A \hto A$,
      we can form a new context $\Gamma \oplus \IdType_x$ by ``adding an homomorphism type at $x$''.
      \[
        \begin{aligned}
          \Gamma \oplus \IdType_x
          =
          \bigodot_{i \in \el{\pi}}
          \begin{cases}
            \Gamma_i & \text{if $i \neq x$}\\
            \IdType_{\Gamma_x} & \text{if $i = x$}
          \end{cases}
        \end{aligned}
      \]
      Similarly there is a canonical reflexivity substitution
      $r^\Gamma_x : \Gamma \to \Gamma \oplus \IdType_x$
      defined by
      \[
        \begin{aligned}
          r^\Gamma_x
          =
          \bigodot_{i \in \el{\pi}}
          \begin{cases}
            \id_i & \text{if $i \neq x$}\\
            r_x & \text{if $i = x$}
          \end{cases}
        \end{aligned}
      \]
      When $\Gamma$ is clear from the context we will denote this substitution by $r_x$.
    \end{definition}

    \begin{definition}
      We say that a globular multicategory has \textbf{homomorphism types}
      when its underlying globular multigraph is reflexive and we have the following  structure:
      \begin{itemize}
       \item For each $n$-term $f : \Gamma \to A$
       and each $(n-1)$-variable $x$,
       there is a \textbf{J-term}
       \[
        \begin{aligned}
          J_x(f) : \Gamma \oplus \IdType_x \to A
        \end{aligned}
      \]
      such that
      \[
        \begin{aligned}
          r_x ; J_x(f) = f
        \end{aligned}
        .\]

        \item Suppose that $0 < k < n$.
        Let $f : \Gamma \to A$ be an $n$-term.
        Then for each $k$-variable $x$
        and any
        \[
          \begin{aligned}
            j_s :  s\Gamma \oplus \IdType_x \to sA\\
            j_t : t\Gamma \oplus \IdType_x \to tA
          \end{aligned}
        \]
        such that
        \[
          \begin{aligned}
            {r_x ; j_s = sf}\\
            {r_x ; j_t = tf}
          \end{aligned}
        \]
        we have a \textbf{J-term}
        \[
          \begin{aligned}
            J^{j_s, j_t}_x(f) : \Gamma \oplus \IdType_x \to A
          \end{aligned}
        \]
        such that
        \[
          \begin{aligned}
            s  J^{j_s, j_t}_x(f) &= j_s\\
            t J^{j_s, j_t}_x(f) &= j_t
          \end{aligned}
        \]
        and
        \[
          \begin{aligned}
            r_x ; J^{j_s, j_t}_x(f) = f
          \end{aligned}
        \]
      \end{itemize}
    \end{definition}

    \begin{example}
      Every type theory with identity types $\typeTheoryT$ induces a globular multicategory $\globT$ with homomorphism types. The construction essentially follows  van den Berg and Garner \cite{Berg2011Types-are-weak-}. We have that:
      \begin{itemize}
        \item A $0$-type $A$ in $\globT$ is a type
        \[
          \begin{aligned}
            \vdash \interpret A : \Type
          \end{aligned}
        \]

        in $\typeTheoryT$.
        \item An $(n+1)$-type $M : A \hto B$ in $\glob T$ is a dependent type judgement
        \[
          \begin{aligned}
            x : \interpret A, y : \interpret B \vdash \interpret M(x, y) : \Type
          \end{aligned}
        \]
        in $\typeTheoryT$.

        \item Each globular context $\Gamma : s\Gamma \hto t\Gamma$ in $\globT$
        corresponds to a list of dependent types in $\typeTheoryT$ and thus induces a dependent context
        \[
          \begin{aligned}
            \interpret {s\Gamma},  \interpret{t \Gamma} \vdash \interpret{\Gamma}
          \end{aligned}
        \]
        in $\typeTheoryT$.

        \item A $0$-term in $f : \Gamma \to A$ in $\globT$ is a term
        \[
          \vec x : \interpret \Gamma \vdash \interpret f \vec x : \interpret A
        \]
        in $\typeTheoryT$.

        \item Suppose that $\Gamma$ is an $(n+1)$-context such that
        $\vec x_s : \interpret{s \Gamma}, \vec x_t : \interpret{t \Gamma} \vdash \interpret  \Gamma(\vec x_s, \vec x_t)$ is a dependent context in $\typeTheoryT$.
        Then an $(n+1)$-term in $f : \Gamma \to A$ in $\globT$ is a term
        \[
          \begin{aligned}
            \vec x : \interpret \Gamma(\vec x_s, \vec x_t)
            \vdash \interpret f \vec x
            : \interpret A(\interpret{sf}(\vec{x_s}), \interpret{tf}(\vec{x_t}))
          \end{aligned}
        \]
        in $\typeTheoryT$.

        \item Composition of terms in $\globT$ is defined by substitution in $\typeTheoryT$.

        \item The homomorphism types, reflexivity terms and J-terms in $\globT$ are the corresponding terms in $\typeTheoryT$.

      \end{itemize}
    \end{example}

    \begin{example}
      In forthcoming work we will construct a globular multicategory with homomorphism types whose:
      \begin{itemize}
        \item $0$-types are strict $\omega$-categories
        \item $(n+1)$-types are profunctors between $n$-types
        \item $0$-terms are strict $\omega$-functors
        \item $(n+1)$-terms are transformations between profunctors
        \item Homomorphism type data comes from a version of the Yoneda Lemma for strict $\omega$-categories
      \end{itemize}
    \end{example}

    \begin{definition}
      A \textbf{globular category} $A$ is a globular object in the category of categories.
      Following Batanin \cite{Batanin1998Monoidal-Globul}
      a \textbf{monoidal globular category} is an $\omega$-category internal to the category of globular categories up to isomorphism.
      This amounts to a globular category $A$ together with:
      \begin{itemize}
        \item for each $k < n$,
        composition functors
        \[
          \begin{aligned}
            \otimes_k : A(n) \times A(n) \longrightarrow A_n
          \end{aligned}
        \]

        \item for each $n$, a unit functor
        \[
          \begin{aligned}
            Z : A(n) \longrightarrow A(n-1)
          \end{aligned}
        \]

        \item natural transformations and axioms, mimicking those of a strict $\omega$-category up to isomorphism.
      \end{itemize}
      When these natural transformations are all identities we say that a globular multicategory is \textbf{strict}.
      We denote the category of strict monoidal globular categories by $\monGlobCat$.
    \end{definition}

    Every monoidal globular category can be viewed as a globular multicategory (see \cite{Cruttwell2010A-unified-frame}). We will show that we always have homomorphism types in this case.

    \begin{prop}
      The globular multicategory induced by a monoidal globular category has homomorphism types.
    \end{prop}
    \begin{proof}[Proof sketch]
      We will work in the strict case.
      There is a coherence theorem for monoidal globular categories (see \cite{Batanin1998Monoidal-Globul}) and so there should be no real loss of generality.
      \begin{itemize}
        \item an  $n$-type is an object of $A(n)$
        \item $n$-term $f : \bigodot_{i} M_i \to N$
        is an arrow
        \[
          \begin{aligned}
            \bigotimes_{i} M_i \to N
          \end{aligned}
        \]
        in $A(n)$.
        Here $\bigotimes$ is repeated application of $\otimes$
        and we follow the convention that a nullary product is a unit.
      \end{itemize}
      This globular multicategory has homomorphism types.
      For any type $M$,
      we set
      \[
        \begin{aligned}
          \IdType_M = Z(M)
        \end{aligned}
      \].
      The $(n+1)$-context $M$ corresponds to the object $Z(M)$
      and so we define the reflexivity term $r_M : M \to \IdType_M$
      to be the identity arrow
      \[
        \begin{aligned}
          \id_{Z(M)} : Z(M) \to Z(M)
        \end{aligned}
      \]
      The J-terms can now be constructed using the coherence laws of $Z$.
      For instance,
      whenever $f : A \to B$ is an $n$-term
      and $x : M$ is a $k$-variable.
      We define
      \[
        \begin{aligned}
          J_x(f) : A \odot_k \IdType_x \to B
        \end{aligned}
      \]
      to be the composite
      \[
        \begin{aligned}
          A \otimes_k Z(M) \longrightarrow A \overset{f}\longrightarrow B
        \end{aligned}.
      \]
      where the arrow on the left is a unit law.
    \end{proof}

    \begin{definition}
      We say that a globular multicategory has \textbf{strict homomorphism types} when for any $n$-terms $f, f'$ and any variable $x$ such that
      \[
        \begin{aligned}
          r_x ; f = r_x ; f'
        \end{aligned},
      \]
      we have that
      \[
        \begin{aligned}
          f = f'
        \end{aligned}
      \]
    \end{definition}

    \begin{example}
      Every strict monoidal globular category induces a globular multicategory with strict homomorphism types.
    \end{example}

    Many more familiar examples can be seen as low-dimensional globular multicategories.

    \begin{definition}
      Let $X$ be a globular multicategory with identity types
      and let $n \geq 0$.
      We say that $X$ is $n$-\textbf{truncated},
      when for any $m$-term $f : \Gamma \to A$ with $m > n$,
      we have that:
      \begin{itemize}
        \item  There exists an $n$-context $\Gamma'$
        and variables $x_1, \ldots, x_l$
        such that
        \[
          \begin{aligned}
           \Gamma = \Gamma' \oplus \IdType_{x_1} \oplus \cdots \oplus \IdType_{x_l}
         \end{aligned}.
       \]

       \item There exists an $n$-type $A'$ such that
       \[
        \begin{aligned}
          A = \IdType^{n-m} A'
        \end{aligned}.
      \]

      \item There exists an $n$-term $f'  : \Gamma' \to A'$  such that
      \[
        \begin{aligned}
          f = J_{x_1}(\cdots (J_{x_l}(f')
        \end{aligned}.
      \]
    \end{itemize}
  \end{definition}

  \begin{example}
    The category of $1$-truncated strict monoidal globular categories is equivalent to the category of double categories.
  \end{example}

  \begin{example}
    The category of $1$-truncated globular multicategories with strict homomorphism types
    is equivalent to
    the category of \emph{virtual double categories} with horizontal units as described by Crutwell and Shulman \cite{Cruttwell2010A-unified-frame}.

    Ibid., the monoids and modules construction for virtual double categories is exhibited as the the right adjoint of the $2$-functor which forgets horizontal units.
    Many familiar collections of ``category-like'' objects  can be seen as   the result of this construction.
    Hence any such collection gives rise to a
    globular multicategory with homomorphism types.

    In future work we will define a \emph{higher modules construction}. This will allow us to study globular multicategories of ``higher category-like'' objects.
  \end{example}

  \begin{example}
    Riehl and Verity \cite{Riehl2017Kan-extensions-} describe
    a  framework for reasoning about ``$(\infty, 1)$-category'-like'' objects
    using structures called $\infty$-cosmoi.
    Every $\infty$-cosmos gives rise to a virtual double category of modules
    whose horizontal units are arrow $\infty$-categories.
    This construction,
    throws away much higher dimensional data
    but still provides a convenient framework for a great deal of formal category theory.
    We expect to be able to construct
    more general globular multicategories with homomorphism types from $\infty$-cosmoi
    which retain this higher dimensional data.
  \end{example}

  \begin{remark}
    Starting with a category with a well behaved notion of path object, van den Berg and Garner \cite{Berg2011Types-are-weak-} construct a globular multicategory with homomorphism types.
    In future work we will construct a globular multicategory with homomorphism types from a category with \emph{directed path objects}.
  \end{remark}

  \section{The Homotopy Theory of Globular Multicategories}\label{homotopy}

  By Remark \ref{globPlusDirect}. the category $\globPlus$ of generic types and terms is a direct category.
  This induces a weak factorization system on globular multicategories and related structures.

  \begin{definition}
    Let us denote by
    \[
      \begin{aligned}
        I = \{ \iota_u : \partial u \longrightarrow u \mid u \in \globPlus\}
      \end{aligned}.
    \]
    the set of \emph{boundary inclusions} of $\globPlus$.
    Then $I$ cofibrantly generates a weak factorization system $(\mathcal{L}, \mathcal{R})$ on $\globGraph$.
    We refer to maps in $\mathcal{L}$ as \textbf{cofibrations} and maps in $\mathcal{R}$ as \textbf{acyclic fibrations}.
    A map of globular multigraphs $f : X \to Y$ is an acyclic fibration
    when, for any generic type or term $u$,
    each commutative square
    \begin{equation*}
     \begin{tikzcd}
      \partial u \ar{r}\ar[swap]{d}{\iota_u} & X\ar{d}{f}\\
      u \ar{r}\ar[dotted]{ur} & Y\\
    \end{tikzcd}
  \end{equation*}
  has a filler.
  A map of globular multigraphs $i : Z \to W$ is a cofibration
  when for each acyclic fibration $f : X \to Y$,
  each commutative square
  \begin{equation*}
   \begin{tikzcd}
    \partial Z \ar{r}\ar[swap]{d}{i} & X\ar{d}{f}\\
    W \ar{r}\ar[dotted]{ur} & Y\\
  \end{tikzcd}
\end{equation*}
has a filler.
\end{definition}

\begin{prop}
  A map of globular multigraphs is a cofibration exactly when it is a monomorphism.
\end{prop}
\begin{proof}
  Since $\globPlus$ is a direct category, it is skeletal and has no non-trivial automorphisms.
  The result now follows from \cite[Proposition 8.1.37]{MR2294028}.
\end{proof}

Our weak factorization system can be transferred to other categories of interest using the adjunctions induced by various forgetful functors
We have the following commutative diagram of functors:
\begin{equation*}
  \begin{tikzcd}[column sep = small]
    &&&
    \globMult \ar{rd}
    &
    \\
    \monGlobCat \ar{r}
    &
    \globMultStrId \ar{r}
    &
    \globMultId \ar{rd} \ar{ru}
    &
    &
    \globGraph
    \\
    &&&
    \rGlobGraph \ar{ru}
    &
  \end{tikzcd}
\end{equation*}
Each of these forgets essentially algebraic data and so has a left adjoint.
Let $U : \mathcal{C} \to \mathcal{D}$ be one of these forgetful functors and let $
F : \mathcal{D} \to \mathcal{C}$ be its left adjoint.
Then the weak factorization system of \textbf{cofibrations} and \textbf{acyclic fibrations} in $\mathcal{C}$ is generated by
\[
  \begin{aligned}
    F \iota_u : F \partial u \longrightarrow F u
  \end{aligned}
\]
for each generating cofibration in $\iota_u$ in $\mathcal{C}$.
A morphism $f : X \to Y$ in $\mathcal{C}$ is an acyclic fibration exactly when $Uf$ is an acyclic fibration in $\mathcal{D}$.
Furthermore $\mathcal{F}$ preserves cofibrations.

\begin{example}\label{operadFibrations}
Suppose that $X$ and $Y$ are globular operads.
Then every homomorphism of globular operads is bijective on types and so the lifting conditions for generic types are always satisfied.
It follows that a homomorphism $f : X \to Y$ is an acyclic fibration
if and only if it satisfies the lifting conditions for generic terms.

We say that a globular operad is \emph{normalized} when it has a unique $0$-term. (This term must be the identity term on the unique $0$-type).
Suppose  that the globular operad $X$ is normalized
and let $\top$ be the terminal globular operad.
Then the canonical map
\begin{equation*}
  \begin{tikzcd}
    X \ar[two heads]{r}{!} & \top
  \end{tikzcd}
\end{equation*}
is an acyclic fibration exactly when $P$ is a normalized contractible globular operad.
This follows from the observations of Garner \cite{Garner2009A-homotopy-theo}.
The algebras of $P$ are weak $\omega$-categories in the sense of Leinster \cite{LeinsterHigher-Operads-}.
\end{example}

The terms defining the natural transformations in this example motivate the following definition:
\begin{definition}
  Suppose that $f,f' : \Gamma \to A$ are parallel $n$-terms in a globular multicategory with identity types.
  Then a \textbf{transformation}
  \[
    \begin{aligned}
      \phi : f \longrightarrow f'
    \end{aligned}
  \]
  is a term $\phi : \Gamma \to \IdType_A$ with $\phi : f \hto f'$.
  Given a term $g : f' \hto tg$,
  the composite
  \[
    \begin{aligned}
      \phi \circ g : f \longhto tg
    \end{aligned}
  \]
  is defined to be $(\phi \odot_n g) ; m$
  where $m = J_{tA}(\id_A)$.
  Given $h : sh \hto f$, we define
  \[
    \begin{aligned}
     h \circ \phi : sh \longhto f'
   \end{aligned}
 \]
 similarly.
 We have that $s (h \circ \phi) = sh$ and $t (h \circ \phi) = f'$.
 The transformation $f ; r_A : f \to f$ is the identity at $f$ with respect to $-\circ-$.

 A transformation between substitutions is a pasting diagram of transformations between terms.
\end{definition}

This definition immediately implies the following lemma. One interpretation of this lemma is  that composition with reflexivity terms defines a sort of weak equivalence.

\begin{lemma}\label{reflexEquiv}
Suppose that $f, f' : \Gamma \oplus \IdType_x \to A$ are terms in a globular multicategory
with identity types.
Then whenever
\[
  \begin{aligned}
   r_x ; f = r ; f'
 \end{aligned}
\]
we have
\[
  \begin{aligned}
   J_x^{f, f'}(r_x ; f ; r_A) : f \longrightarrow f'
 \end{aligned}
\]
When we have strict homomorphism types,
we have that $f = f'$
and $J_x^{f, f'}(r_x ; f ; r_A) = f ; r_A$.
\end{lemma}

Transformations can be used to provide a useful alternative description of acyclic fibrations.
Intuitively this description says that term-lifting properties of acyclic fibrations are satisfied exactly when, on terms, a homomorphism is``strictly surjective and weakly reflects identities''.

\begin{definition}
  Let $f : X \to Y$ be a homomorphism of globular multicategories with identity types.
  We say that $f$ \textbf{weakly reflects identities} of terms
  if whenever $v, v' : \Gamma \to A$ are parallel terms in $X$ such that $f(v) = f(v')$,
  we have a transformation $\phi : v \to v'$ such that
  $f(\phi) = f(v) ; r_A$.
  In this case we say that $\phi$ is an \textbf{identification}.
  We say that $f$ \textbf{strictly reflects identities} when all the corresponding identifications can be chosen to be identity transformations.
\end{definition}

\begin{prop}\label{acyclicFibrationsAlternative}
A homomorphism of globular multicategories with homomorphism types $f : X \to Y$ is an acyclic fibration if and only if all the following conditions hold:
\begin{enumerate}[label=(\roman*)]
  \item The homomorphism $f$ has the right lifting property against the boundary-inclusions of types. \label{typeLifting}
  \item The homomorphism $f$ is surjective on terms. \label{surjective}
  \item The homomorphism $f$ weakly reflects identities of terms. \label{faithful}
\end{enumerate}
\end{prop}
\begin{proof}
  First suppose that $f$ is an acyclic fibration. Then \ref{typeLifting} follows trivially.
  For each generic type or term $u$,
  the unique map $\emptyset \to u$ is a cofibration.
  The lifting property of $f$ with respect to this map tells us that $f$ is surjective on types or terms with the same shape as $u$.
  This proves \ref{surjective}.

  Now suppose that $v, v' : \Gamma \to A$ are parallel terms in $X$ and that $f(v) = f(v')$.
  Let $u$ be the generic term with the same shape as $v$ and $v'$.
  Let $I_u$ be the generic term with same shape as transformations between $v$ and $v'$.
  Then$v$ and  $v'$ together correspond to a map $\boundaryMap{v}{v'} : \partial I_u \to X$ of globular multigraphs.
  Furthermore, we have the following commutative square in $\globGraph$:
  \begin{equation*}
    \begin{tikzcd}
      \partial I_u
      \ar[r, "\boundaryMap{v}{v'}"]
      \ar[d, "\partial I_u" left]
      &
      X
      \ar[d, "f"]
      \\
      I_u \ar[ur, dotted]
      \ar[r, "f(v);r_A" below]
      &
      Y
    \end{tikzcd}
  \end{equation*}
  Since $f$ is an acyclic fibration,
  this square has a filler.
  This filler defines the transformation $v \to v'$ required by \ref{faithful}.

  Now suppose on the other hand that we have \ref{typeLifting}, \ref{surjective} and \ref{faithful}.
  Let $u$ be a generic  term and fix a commutative square:
  \begin{equation*}
   \begin{tikzcd}
    \partial u \ar{r}{\widetilde{\partial v}}
    \ar[swap]{d}{\iota_u} & X\ar{d}{f}\\
    u \ar[swap]{r}{v} & Y\\
  \end{tikzcd}
\end{equation*}
Let the source and target of $\widetilde {\partial v}$ be $\widetilde{sv}$ and $\widetilde{tv}$ respectively.
By \ref{surjective},
there is a term $w$ in $X$ with $f(w) = v$.
It follows that
$f(sw) = sv = f(\widetilde{sv})$
and $f(tw) = tv = f(\widetilde{tv})$.
Hence, by \ref{faithful} there are transformations
$\phi : \widetilde{sv} \to sw$ and $\psi : tw \to \widetilde{tv}$
such that $f(\phi) = sw ;r$  and $f(\psi) = tw ; r$.
We define
\[
 \begin{aligned}
   \tilde v = \phi \circ w \circ \psi
 \end{aligned}
\]
By construction $\partial \tilde v = \widetilde{\partial v}$.
Furthermore, every homomorphism of globular multicategories with homomorphism types preserves $- \circ -$ and so $f(\tilde v) = f(w) = v$.
Hence $\tilde v$ defines the required filler and so $f$ is an acyclic fibration.
\end{proof}

Our next result describes conditions under which, given a category $\mathcal{C}$ and \emph{strictification} functor $S$ which respects homotopy theoretic information in $\mathcal{C}$, there is an induced acyclic fibration of endomorphism operads.

\begin{theorem}\label{endAcyclic}
Let $\mathcal{C}$ be a category with a factorization system $(\mathcal{L}, \mathcal{R})$.
Let $X : \glob^\op \to \mathcal{C^\op}$ be a globular object in $\mathcal{C^\op}$
and let $S : C \to D$ be a pullback-preserving functor.
Suppose that there exists a functor $U : \mathcal{D} \to \mathcal{C}$
together with a natural transformation $\phi : \id \Rightarrow US$.
Suppose that all the following conditions hold:
\begin{itemize}
  \item The functor $U$ is faithful.
  \item Each boundary inclusion $X(\iota_n) : X(\partial n) \to X(n)$
  is an $\mathcal{L}$-map.
  \item For each globular pasting diagram $\pi$, the arrow $\phi_{X(\pi)}$ is an $\mathcal{R}$-map and an epimorphism.
\end{itemize}
Then the induced homomorphism of globular operads
\[
 \begin{aligned}
   \End(S^\op) : \End(X) \longrightarrow \End(S^\op X)
 \end{aligned}
\]
is an acyclic fibration of globular multicategories.
\end{theorem}

\begin{proof}
  Following Example \ref{operadFibrations} it suffices to check the lifting conditions for generic terms.

  Let $u$ be the generic $\pi$-shaped $n$-term
  and suppose that we have a commutative square of the following form:
  \begin{equation*}
    \begin{tikzcd}
      \partial u \ar{r}{\widetilde {\partial v}} \ar[swap]{d}{\iota_u}
      &
      \End(X)  \ar{d}{\End(S)}
      \\
      u \ar[swap]{r}{v} & \End(S^\op X)
    \end{tikzcd}
  \end{equation*}

  By the Yoneda Lemma,
  the homomorphism $\widetilde \partial v$ corresponds to a term boundary
  $\widetilde{\partial f}$ in $\End(X)$

  and the homomorphism $v$
  corresponds to a term  $f$ in $\End(S^\op X)$.
  Commutativity of the square tells us that
  \[
    \begin{aligned}
      S(\widetilde{\partial f}) = \partial f
    \end{aligned}.
  \]
  Expanding the definition of $\End$,
  we find that
  $\widetilde {\partial f}$ corresponds to
  a  boundary-preserving arrow $\widetilde{\partial f} : X(\partial n) \to X(\pi)$ in $\mathcal{C}$.
  (Here  $X(\pi)$ comes from the Yoneda extension of  $X$.)
  Similarly $f$ corresponds to
  a boundary-preserving arrow $f : SX(n) \to SX(\pi)$ in $\mathcal{D}$.
  Hence we have the following commutative diagram in $\mathcal{C}$:
  \begin{equation*}
    \begin{tikzcd}
      &
      X(\partial n)
      \ar{r}{\widetilde{\partial f}}
      \ar[swap]{d}{\phi_{X(\partial n)}}
      &
      X(\pi)
      \ar{d}{\phi_{X(\pi)}}
      \\
      X(\partial n)
      \ar{r}{\phi_{X(\partial n)}}
      \ar[swap]{d}{X(\iota_n)}
      \ar[ur, equal]
      &
      USX(\partial n)
      \ar[r, bend left = 20, "{US(\widetilde{\partial(f)})}"{name = top,  above}]
      \ar[r, bend right = 20, "{U \partial f}"{below,  name=bottom}]
      \ar[swap]{d}{USX(\iota_n)}
      \
      \ar[equal, from = top, to = bottom, shorten <= 1ex, shorten >= 1ex]
      &
      USX(\pi)
      \\
      X(n)
      \ar[swap]{r}{\phi_{X(n})}
      &
      USX(n)
      \ar[ur, bend right = 30, "U f" swap]
    \end{tikzcd}
  \end{equation*}
  The squares commute by naturality
  and the bottom triangle commutes by definition of $\partial f$.
  Now, since  $X(\iota_n)$ is an $\mathcal{L}$-map
  and $\phi_{\pi}$ is an $\mathcal{R}$-map,
  the outer rectangle has a filler
  \begin{equation*}\tag{$\dagger$}\label{internalFiller}
  \begin{tikzcd}
    X(\partial n)
    \ar{r}{\widetilde {\partial f}}
    \ar[swap]{d}{X(\iota_n)}
    &
    X(\pi)
    \ar{d}{\phi_{X(\pi)}}
    \\
    X(n)
    \ar[swap]{r}{Uf  \circ \phi_{X(n)}}
    \ar[ur, dashed, "{\tilde f}" description]
    &
    USX(\pi)
  \end{tikzcd}
\end{equation*}
We know that $\tilde f$ is boundary preserving because its boundary $\widetilde{\partial f}$ is boundary-preserving.
Hence, using the definition of $\End$, we can view $\tilde f$ as a $\pi$-shaped $n$-term in $\End(X)$.
By the Yoneda Lemma $\tilde f$ corresponds to a map of globular multigraphs
\[
  \begin{aligned}
    \tilde v : u \longrightarrow \End(X)
  \end{aligned}.
\]
Hence it remains to show that the following diagram commutes:
\begin{equation*}\tag{$\star$}\label{externalFiller}
\begin{tikzcd}
  \partial u \ar{r}{\widetilde {\partial v}} \ar[swap]{d}{\iota_u}
  &
  \End(X)  \ar{d}{\End(S^\op)}
  \\
  u \ar[swap]{r}{v} \ar[ur, dashed, "\tilde v" description]& \End(S^\op X)
\end{tikzcd}
\end{equation*}
Commutativity of the top triangle of (\ref{externalFiller}) says that
\[
  \begin{aligned}
    \partial \tilde f = \widetilde{\partial f}
  \end{aligned}.
\]
This follows from commutativity of the top triangle of (\ref{internalFiller}).
Commutativity of the bottom triangle of (\ref{externalFiller})
says that
\[
  \begin{aligned}
    S(\tilde f) = f
  \end{aligned}.
\]
By naturality of $\phi$
and commutativity of the bottom triangle of (\ref{internalFiller}),
we have that
\[
  \begin{aligned}
   US(\tilde f) \circ \phi_n=  \phi_{\pi} \circ \tilde f= Uf \circ \phi_{n}
 \end{aligned}.
\]
Since $\phi_n$ is an epimorphism and $U$ is faithful, the result now follows.
\end{proof}

\section{Weak Higher Categorical Structure}\label{sectionWeakHigherCategories}

In this section we show that the types and terms of globular multicategories with homomorphism types have higher categorical structure.
We will demonstrate two related results. Firstly each piece of data \emph{internal} to a globular multicategory with homomorphism types has higher categorical structure. For example, each $0$-type has the structure of a weak $\omega$-category and each $0$-term has the structure of a weak $\omega$-functor.
We will then combine this internal data and show that the \emph{external} collection of types and terms has the structure of a weak $\omega$-category.
Most of the work involved in moving from the internal to the external situation is done by Theorem \ref{endAcyclic}.

For the remainder of this section we will label the adjunctions between reflexive globular multigraphs, globular multicategories with homomorphism types and globular multicategories with strict homomorphism types as in the following diagram.

\begin{equation*}
  \begin{tikzcd}
    \monGlobCat
    \ar[r, bend right,  "V" below]
    \ar[r, phantom, "\perp" description]
    &
    \ar[bend right, l, "\FreeCatStr" above]
    \globMultStrId \ar[r, bend right,  "U" below]
    \ar[r, phantom, "\perp" description]
    &
    \globMultId
    \ar[bend right, l, "\strictify" above]
    \ar[bend right, r]
    \ar[r, phantom, "\perp" description]
    &
    \rGlobGraph \ar[bend right, l, "\FreeH" above]
  \end{tikzcd}
\end{equation*}
The leftmost adjunction adds \emph{composition of types}.
We denote the unit of this adjunction by $\beta : \id \Rightarrow V \FreeCatStr$.
The middle adjunction can be thought of as describing \emph{strictification} of globular multicategories with homomorphism types.
We denote the unit of this adjunction by  $\etaStr : \id \Rightarrow U \strictify$.
We denote the unit of the composite adjunction $\FreeCatStr S \dashv UV$ by $\theta : \id \Rightarrow UV \FreeCatStr S$.
We write $\FreeCat = \FreeCatStr S \FreeH$.
Our ``internal'' result can now be stated precisely.

\begin{theorem}\label{internalStrucutre}
Let $X$ be a reflexive globular multigraph.
Then the unit of the strictification adjunction at $\FreeH X$
\begin{equation*}
  \begin{tikzcd}
    \FreeH X \ar[two heads, swap]{d}{\etaStr_{\FreeH X}}
    \\
    U \strictify \FreeH X
  \end{tikzcd}
\end{equation*}
is an acyclic fibration of globular multicategories with homomorphism types.
Furthermore, $\eta_X$ is bijective on types and $0$-terms.
\end{theorem}

\begin{proof}
  We first describe the strictification adjunction $S \dashv U$ explicitly.
  The globular multicategory $\strictify Y$ is the result of forcing $Y$ to satisfy the additional requirement that
  for any $n > 0$, $n$-terms $f, f'$ and any variable $x$ such that
  \[
    \begin{aligned}
      r_x ; f = r_x ; f'
    \end{aligned},
  \]
  we have that
  \[
    \begin{aligned}
      f = f'
    \end{aligned}.
  \]
  Hence,
  for every pair of contexts  $\Gamma, \Delta$,
  we define an equivalence relation $\sim$
  relating pairs of substitutions $f, f' : \Gamma \to \Delta$
  using the following inductively defined rules:
  \begin{itemize}
    \item This relation is reflexive, symmetric and transitive.

    \item For any variable $x$ and terms $f, f'$,
    we have that
    \[
      \begin{aligned}
        r_x ; f \sim r_x ; f' \implies f \sim f'
      \end{aligned}.
    \]

    \item Whenever we have substitutions $f, f', g, g'$ such that
    $t_k f = s_k g$ and $t_k f' = s_k g'$,
    we have that
    \[
      \begin{aligned}
        f  \sim f' &&\text{and}&& g \sim g' \implies f \odot_k g \sim f' \odot_k g'
      \end{aligned}.
    \]

    \item Whenever $f : \Gamma \to \Delta$ and $g : \Delta \to A$,
    we have that
    \[
      \begin{aligned}
        f  \sim f' &&\text{and}&& g \sim g' \implies f ; g \sim f' ; g'
      \end{aligned}.
    \]
  \end{itemize}
  Since terms are the same as substitutions whose target context is a type,
  this defines an equivalence relation on terms.
  The globular multicategory $\strictify Y$ has the same types as $Y$ and has equivalence classes of terms under $\sim$ as terms.
  The rules defining $\sim$ ensure that the composition and $J$-rules of $\strictify Y$ are well defined and that the homomorphism types are strict.
  Furthermore, the unit $\eta_Y : Y \to U \strictify Y$ is the homomorphism that quotients by $\sim$.

  Using this description of $\sim$ we see that pointwise $\eta$ is bijective on types, bijective on $0$-terms and surjective on terms.
  Hence by Proposition \ref{acyclicFibrationsAlternative},
  it suffices to show that $\eta_{\FreeH X}$  weakly reflects identities of terms.
  Lemma \ref{reflexEquiv} will do most of the work for us and allow us to construct a sort of ``normalization procedure'' for terms in $\FreeH X$.
  The resulting ``normal form'' actually lives in the monoidal globular category $\FreeCat X$. In other words, by plugging in reflexivity terms, we can always obtain a formal composite of generators in $X$.

  Consider the following commutative diagram:
  \begin{equation*}
    \begin{tikzcd}
      \FreeH X \ar[two heads, swap]{d}{\etaStr_{\FreeH X}}
      \ar[r, equal]
      &
      \FreeH X
      \ar[d, two heads, "\theta_{\FreeH X}"]
      \\
      U \strictify \FreeH X
      \ar[r, hook, "U \beta_{\strictify \FreeH X}" below]
      &
      U V \FreeCatStr S \FreeH X
      \ar[r, equal]
      &
      UV \FreeCat X
    \end{tikzcd}
  \end{equation*}
  Since $V$ is faithful and $U$ preserves limits,
  the bottom arrow is monic.
  It follows that $\eta_{\FreeH X}$ weakly reflects identities of terms
  if $\theta_{\FreeH X}$ weakly reflects identities of terms.
  This is Lemma \ref{monGlobCatLemma}\ref{weakIdentities} below.
\end{proof}

\begin{lemma}\label{monGlobCatLemma}
Let $X$ be a reflexive globular multigraph. Then we have the following results:
\begin{enumerate}[label=(\roman*)]
  \item The unit \label{weakIdentities}
  \begin{equation*}
    \begin{tikzcd}
      \FreeH X
      \ar[d, two heads, "\theta_{\FreeH X}"]
      \\
      UV \FreeCat X
    \end{tikzcd}
  \end{equation*}
  weakly reflects identities of terms.

  \item The unit \label{strictIdentities}
  \begin{equation*}
    \begin{tikzcd}
      S \FreeH X
      \ar[d, two heads, "\beta_{S \FreeH X}"]
      \\
      V G X
    \end{tikzcd}
  \end{equation*}
  strictly reflects identities of terms.
  The furthermore part is easily verified.
\end{enumerate}
\end{lemma}
\begin{proof}
  We will prove \ref{weakIdentities}.
  The proof of \ref{strictIdentities} is similar except all identifications are actual identities.

  Let $X$ be a reflexive globular set.
  Then explicitly:

  \begin{itemize}
    \item $\FreeH X$ is generated from $X$ by freely adding $-;-$ composition, reflexivity terms and J-terms.

    \item $UV \FreeCat X$ is freely generated from $X$ by freely adding $- ; -$ composition and $-\otimes_k$- composition for each $k$.
    Equivalently a term in $UV \FreeCat X$ is a formal composite $\gamma_1 ; \cdots ; \gamma_l$ of substitutions in $X$.

    \item The homomorphism $\theta_{\FreeH X}$ maps types in $\FreeH X$ to types in  $UV \FreeCat X$  up to ``removing homomorphism types''.
    That is,
    for each  $n$-type $A$, we set
    \[
      \begin{aligned}\
      [A] \sim \IdType_A
    \end{aligned}.
  \]
  Substitutions in $\FreeH X$ are mapped to substitutions in $X$ up to ``removing homomorphism types''.  That is,
  for each term $f$, we set
  \[
    \begin{aligned}\
    [f] \sim [f ; r_A]
  \end{aligned}
\]
and for each variable $x$ and any J-rule at $x$, we set
\[
  \begin{aligned}\
  [f] \sim [J_x(f)]
\end{aligned}.
\]
We also require that $-;-$ and $-\odot_k-$ are respected.
Since terms are substitutions whose target is a type, this defines a homomorphism.

\end{itemize}

Now let $f, f' : \Gamma \to \Delta$ be substitutions in $\FreeH X$.
such that $\theta_{\FreeH X}(f) = \theta_{\FreeH X}(f')$.
Let $H$ be the set of homomorphism type variables in $\Gamma$.

As stated above, the equivalence class $\theta_{\FreeH X}(f)$ corresponds to a composable sequence
\[
  \begin{aligned}
    \gamma_1 ; \gamma_2  ; \cdots ; \gamma_l
  \end{aligned}
\]
of substitutions in $X$.
We will now prove the claim by induction on $l$.

First suppose that $l = 0$.
By induction on $f$,
we must have that
\[
  \begin{aligned}
    r_{H} ; f = R
  \end{aligned}
\]
where $R$ is some composite of identity and reflexivity terms.
Similarly
\[
  \begin{aligned}
    r_{H} ; f = R'
  \end{aligned}
\]
where $R'$ is some composite of identity and reflexivity terms.
Since $f$ and $f'$ have the same target, namely $\Delta$, it follows that $R = R'$.
Hence, since the source contexts of $f$ and $f'$ are the same,
we have an identification $f \to f'$ by Lemma \ref{reflexEquiv}.

Now suppose that $l > 0$.
Let $H_{\gamma_1}$ be the set of homomorphism type variables in the source context of $\gamma_1$.
By induction on the rules defining $\theta$, it follows that
\[
  \begin{aligned}
    r_H ; f = r_{H_{\gamma_1}} ; \gamma_1 ; f_{>1}
  \end{aligned}
\]
for some term $f_{>1}$ such that $\theta_{\FreeH X}(f_{>1}) = \gamma_2 ; \cdots ; \gamma_l$.
Similarly we have that
\[
  \begin{aligned}
    r_H ; f' = r_{H_{\gamma_1}} ; \gamma_1 ; f'_{>1}
  \end{aligned}
\]
for some term $f'_{>1}$ such that $\theta_{\FreeH X}(f'_{>1}) = \gamma_2 ; \cdots ; \gamma_l$.
Now, by induction,
we have an identification
\[
  \begin{aligned}
    \phi : f_{>1} \longrightarrow f'_{>1}
  \end{aligned}
\]
Hence we have an identification
\[
  \begin{aligned}
    J_H^{f, f'}(r_{H_{\gamma_1}} ; \gamma_1 ; \phi) : f \longrightarrow f'
  \end{aligned}.
\]
The result follows by induction.
\end{proof}

\begin{example}
  Let $X$ be a globular multicategory with strict homomorphism types
  and Let $0$ denote the generic $0$-type.
  Then it  follows from our explicit descriptions that $U \strictify \FreeH 0$ is the terminal globular operad.
  Thus, Theorem \ref{internalStrucutre} tells us  that $\FreeH 0$ is a normalized contractible globular operad.
  By the Yoneda Lemma and adjointness, every $0$-type $A$ in $X$ corresponds,
  to a homomorphism $A : \FreeH 0 \to X$ of globular multicategories with homomorphism types.
  It follows that every $0$-type in $X$ with its tower of homomorphism types has the structure of a weak $\omega$-category.
\end{example}

\begin{example}
  Let $I_0$ be the generic $0$-term.
  Then $\strictify \FreeH I_0$ can be seen as the algebraic theory describing a strict
  functor between strict $\omega$-categories.
  The fact that there is an acyclic fibration $\FreeH I_0 \to U \strictify \FreeH I_0$ can be viewed as saying that $\FreeH I_0$ is a weak functor between weak $\omega$-categories.
  Consequently every $0$-term in a globular multicategory with homomorphism types has the structure of a weak functor.
\end{example}

More generally,
Theorem \ref{internalStrucutre} can be seen as saying that,
in a globular multicategory with homomorphism types,
\begin{itemize}
  \item The $0$-types are weak $\omega$-categories.
  \item The $1$-types are weak profunctors.
  \item The $2$-types are weak profunctors between weak profunctors.
  \item \ldots
  \item The $0$-terms are weak $\omega$-functors
  \item The $1$-terms are weak transformations between profunctors
  \item The $2$-terms are weak transformations between $2$-types
  \item \ldots
\end{itemize}

Recall that a transformation between parallel $n$-terms $f, f' : \Gamma \to A$
is a term $\phi : \Gamma \to \IdType_A$ such that
$s \phi = f$ and $t\phi = f'$.
In other words,
a transformation is an abstract assignment taking an object in $\Gamma$ and outputting an arrow in $A$ from the output of $f$ to the output of $f'$.
Hence transformations between $0$-terms,
can be seen as natural transformations.
Similarly transformations between transformations between $0$-terms can be seen as modifications.
Continuing in this way, let $\mathbb{T}_0 : \glob^\op \to \rGlobGraph^\op$ be the globular object defined by:
\[
 \begin{aligned}
   \mathbb{T}_0(n) &=
   \begin{cases}
     \text{the generic $0$-type } \bullet & \text{if  $n = 0$}\\
     \text{the generic term } u_{n-1} : A \to \IdType^{n - 1} B & \text{if $n > 0$}
   \end{cases}
   \\
   \mathbb{T}(\sigma_n) & =
   \begin{cases}
     \text{the inclusion of } A & \text{if  $n = 0$}\\
     \text{the inclusion of }s u_n  & \text{if $n > 0$}
   \end{cases}
   \\
   \mathbb{T}(\tau_n)  &=
   \begin{cases}
     \text{the inclusion of } B & \text{if  $n = 0$}\\
     \text{the inclusion of }t u_n & \text{if $n > 0$}
   \end{cases}
 \end{aligned}
\]
For any globular multicategory with homomorphism types $X$, we have a globular set $n \mapsto \hom(\FreeH^\op \arrowTheoryT(n), X)$ of $0$-types, $0$-terms and transformations between them in $X$.
We will show that this globular set has the structure of a weak $\omega$-category.

\begin{theorem}
  The operad $\End(\FreeH^\op \arrowTheoryT)$ is normalized and contractible.
\end{theorem}
The following result is an immediate application.
\begin{corollary}
  The $0$-types, $0$-terms and transformations between them in a globular multicategory with homomorphism types form a weak $\omega$-category.
\end{corollary}

First note that Theorem \ref{endAcyclic} gives us the following result:

\begin{prop}
  The homomorphism $\End(\FreeH^\op \arrowTheoryT) \to \End(\strictify^\op \FreeH^\op \arrowTheoryT)$ is an acyclic fibration of normalized globular multicategories.
\end{prop}
\begin{proof}
  The globular object $ \FreeH^\op \arrowTheoryT$,
  together with the adjunction
  \begin{equation*}
    \begin{tikzcd}
      \globMultId^\op
      \ar[r, bend left, "\strictify" above]
      \ar[r, phantom, "\bot" description]
      &
      \globMultStrId \ar[l, bend left,  "U" below]
    \end{tikzcd}
  \end{equation*}
  and the natural transformation $\eta :  \id \Rightarrow  U \strictify$ are easily seen to satisfy the conditions of Theorem \ref{endAcyclic}.
\end{proof}

Thus, it suffices to show that $\End(\strictify^\op \FreeH^\op \arrowTheoryT)$ is the terminal globular operad.
This is really an abstract expression of the well known fact that the collection of all strict $\omega$-categories is a strict $\omega$-category.

\begin{lemma}
  Suppose that $X$ is a reflexive globular multigraph and that $Y = S \FreeH X$.
  Let $0 \leq k < n$ and suppose that $\beta_Y(f), \beta_Y(g)$ are parallel terms in $V G Y$
  such that $t_k \beta_Y(f) = s_k \beta_Y(g)$.
  Then we can always choose substitutions $f, g$ in $Y$ such that $t_k f = s_k g$.
  In this case we have that $\beta_Y(f \odot_k g) = \beta_Y(f) \otimes \beta_Y(g)$.
\end{lemma}
\begin{proof}
  It follows from our analysis in the proof of \ref{monGlobCatLemma} that there exist substitutions $\gamma_1, \ldots, \gamma_l$, $\delta_1, \ldots, \delta_l$ in $X$ such that $\beta_Y(f) = \gamma_1 ; \cdots ; \gamma_l$
  and $\beta_Y(g) = \delta_1 ; \cdots ; \delta_{l}$
  and $t_k \gamma_i = s_k \delta i$
  for each $i$.
  We can now use J-terms to find a composable list of substitutions $\epsilon_1, \ldots, \epsilon_l$ in $Y$ such that $\beta_Y(\epsilon_i) = t_k \gamma_i = s_k \delta_i$.
  We can then use J-terms again to find composable lists $\tilde \gamma_1, \ldots \tilde \gamma_l$ and $\tilde \delta_1, \ldots, \tilde \delta_l$ in $Y$ such that $\beta_Y(\tilde \gamma_i) = \gamma_i$ and $\beta_Y(\tilde \delta_i) = \delta_i$ and such that $t_k \tilde \gamma_i = \epsilon_i = t_k \tilde \delta_i$.
  Thus we can set $f = \tilde \gamma_1 ; \cdots \tilde \gamma_l$ and
  $g = \tilde \delta_1 ; \cdots \tilde \delta_l$.
\end{proof}

\begin{definition}
  Suppose that $X$ and $Y$ are as above.
  Suppose furthermore that every term of $X$ is of the form $f : A \to \IdType^n_B$ where $A$ and $B$ are $0$-types.
  In this case we say that $Y$ is an \textbf{arrow theory}.
  We denote the category of arrow theories and homomorphism type preserving maps by $\arrowTheory$ .
\end{definition}
\begin{prop}\label{arrowComposites}
Suppose that $f, g$ are terms in an arrow theory $Y$ such that $t_k f = s_k g$.
Then there exists a term $h$ in $Y$ such that
\[
  \begin{aligned}
    \beta_Y(h) = \beta_Y(f) \otimes \beta_Y(g)
  \end{aligned}.
\]
\end{prop}
\begin{proof}
  We must that $f : \Gamma \to \IdType^n_A$ and $g : \Delta \to \IdType^m_A$ are substitutions where $A$ is a $0$-type in $Y$ and $0 \leq n \leq m$.
  Suppose that $0 \leq k < m$.
  Then using J-rules we can construct a term
  \[
    m :  \IdType^n_A \odot_k  \IdType^m_A \to \IdType^m_A
  \]
  such that $\beta_Y(m) = \beta_Y(r^m_A)$.
  It follows that
  \[
    \begin{aligned}
      \beta_Y((f \odot_k g) ; m) =  (\beta_Y(f) \odot_k \beta_Y(g)) ; \beta_Y(m)
      = \beta_Y(f) \odot_k \beta_Y(g)
    \end{aligned}.
  \]
  and this is the same as $\beta_Y(f) \otimes_k \beta_Y(g)$.
  Hence $\beta_{Y}$ is surjective on types and terms.
\end{proof}
\begin{corollary}\label{classifyArrowTheories}
We have a fully faithful functor
\[
  \begin{aligned}
   \mathbb I: \arrowTheory \longrightarrow \StrOmegaCat
 \end{aligned}.
\]
\end{corollary}
\begin{proof}
  It follows from Lemma \ref{monGlobCatLemma}\ref{strictIdentities} that $\beta_{Y}$ is injective on types and terms.
  Hence the homomorphism $\beta_{Y}$ is bijective on types and terms.
  It follows that there is an induced strict monoidal globular category structure on $Y$ such that
  $\beta : Y \to VG Y$ is an isomorphism of strict monoidal globular categories.
  Moreover every map preserving homomorphism types preserves the  $-\otimes_k-$  operations defined using Lemma \ref{classifyArrowTheories}.
  Hence,
  we have a fully faithful functor
  \[
   \begin{aligned}
    B : \arrowTheory \longrightarrow \monGlobCat
  \end{aligned}.
\]
The monoidal globular categories in the image of $B$ are precisely those with no non-trivial $n$-types for $n > 1$ and which are freely generated by a reflexive globular multigraph.
However, such monoidal globular multicategories are the same as strict $\omega$-categories freely generated by reflexive globular sets.
(We view the $0$-types of an arrow theory as the $0$-cells of a reflexive globular set.
We view the $n$-terms of an arrow theory as the $(n+1)$-cells of this reflexive globular set.)
Hence, we have a fully faithful functor $\mathbb I : \arrowTheory \to \StrOmegaCat$
to the category of strict of $\omega$-categories.
\end{proof}

\begin{prop}
 The globular operad $\End(\strictify^\op \FreeH^\op \arrowTheoryT)$ is the terminal globular operad.
\end{prop}
\begin{proof}
  By definition a $\pi$-shaped $n$-term in this operad is a boundary preserving  homomorphism of arrow theories
  \[
    \begin{aligned}
      v : \strictify \FreeH \arrowTheoryT(n) \longrightarrow \strictify \FreeH \arrowTheoryT(\pi)
    \end{aligned}
  \]
  Using the fully faithful functor  $\mathbb I : \arrowTheory \longrightarrow \StrOmegaCat$ of Corollary \ref{classifyArrowTheories},
  this homomorphism of arrow theories corresponds to
  a boundary preserving functor
  \[
    \begin{aligned}
      \mathbb I(v) : T n \longrightarrow T \pi
    \end{aligned}.
  \]
  from the strict $\omega$-category generated by a single $n$-cell to the strict $\omega$-category generated by the pasting diagram $\pi$.
  In the terminology of \cite{RuneHaugseng2018Oteb}, the functor $\mathbb I(v)$ is \emph{active}.
  It follows that there is a unique functor of this form and so $v$ is uniquely determined.
\end{proof}

\section*{Acknowledgements}
I would like thank my supervisor Kobi Kremnitzer for many useful conversations.
This work at the University of Oxford was supported by an EPSRC studentship.

\bibliography{ms.bib}
\bibliographystyle{plain}

\end{document}